\documentclass{amsart}
\usepackage{amsmath,amssymb}

\newtheorem{theorem}{Theorem}

\newtheorem{corollary}{Corollary}
 
\theoremstyle{remark}
\newtheorem{remark}{Remark}

\newcommand {\bbC} {\mathbb{C}}
\newcommand {\bbF} {\mathbb{F}}
\newcommand {\bbN} {\mathbb{N}}
\newcommand {\bbQ} {\mathbb{Q}}
\newcommand {\bbR} {\mathbb{R}}
\newcommand {\bbZ} {\mathbb{Z}}

\newcommand {\rmN} {\mathrm{N}}
\newcommand {\calP} {\mathcal{P}}
\newcommand {\pg} {\mathfrak{p}}

\newcommand {\Res}{\mathop{\mathrm{Res}}}
\newcommand {\ch}{\mathop{\mathrm{ch}}}
\newcommand {\sh}{\mathop{\mathrm{sh}}}
\renewcommand {\Re}{\mathop{\mathrm{Re}}}
\renewcommand {\Im}{\mathop{\mathrm{Im}}}
\newcommand{\N}{\mathrm{N}}

\title{On logarithmic derivatives of zeta functions in families of global fields}
\author{Philippe Lebacque}
\address{
Philippe Lebacque
\newline \indent
School of Mathematical Sciences of the University of Nottingham 
}
\email{philippe.lebacque@nottingham.ac.uk }

\author{Alexey Zykin}
\address{
Alexey Zykin
\newline \indent
Institut de Math\'ematiques de Luminy
\newline \indent
Mathematical Institute of the Russian Academy of Sciences
\newline \indent
Laboratoire Poncelet (UMI 2615)
}
\email{zykin@iml.univ-mrs.fr}
\date{}
\thanks{The first author was partially supported by EPSRC grant EP/E049109  "Two dimensional
adelic analysis". The second author was partially supported by the grants RFBR 07-01-92211-CNRSa, RFBR 07-01-00051a and RFBR 08-07-92495-CNRSa.}

\begin{document}
\begin{abstract}
We prove a formula for the limit of logarithmic derivatives of zeta functions in families of global fields with an explicit error term. This can be regarded as a rather far reaching generalization of the explicit Brauer--Siegel theorem both for number fields and function fields.
\end{abstract}
\maketitle
\section{Introduction}
The goal of this paper is to prove a formula for the limit of logarithmic derivatives of zeta functions in families of global fields (assuming GRH in the number field case) with an explicit error term. This result is close in spirit both to the explicit Brauer--Siegel and Mertens theorems from \cite{Leb} as well as to the generalized Brauer--Siegel type theorems from \cite{Zyk}. We also improve the error term in the explicit Brauer--Siegel theorem from \cite{Leb}, allowing its dependence on the family of global fields under consideration.

Throughout the paper the constants involved in $O$ and $\ll$  are absolute and effective (and, in fact, not very large). Let $K$ be a global field that is a finite extension of $\bbQ$ or a finite extension of $\bbF_r(t),$ in the latter case $K=\bbF_r(X)$ for a smooth absolutely irreducible projective curve over $\bbF_r$, where $\bbF_r$ is the finite field with $r$ elements. We will often use the acronyms NF or FF for the statements proven in the number field and the function field cases respectively. By $\log(x)$ we understand the natural logarithm $\ln(x)$ in the number field case and $\log_r(x)$ in the function field case. We shall often omit the index $K$ in our notation in cases when it creates no confusion.

For a number field $K$ let $n_K$ and $D_K$ denote its degree and its discriminant respectively. Let $g_K$ be the genus of a function field, that is the genus of the corresponding smooth projective curve and let $g_K=\log \sqrt{D_K}$ in the number field case. Let $\calP(K)$ be the set of (finite) places of $K$ and let $\Phi_{q}=\Phi_{q}(K)$ be the number of places of norm $q$ in $K$, i. e. $\Phi_{q}=|\{\pg\in \calP(K)| \rmN \pg=q\}|.$ In the number field case we denote by $\Phi_\bbR=r_1$ and $\Phi_\bbC=r_2$ the number of real and complex places of $K$ respectively.

Recall that the zeta function of a global field $K$ is defined as $$\zeta_{K}(s)=\prod\limits_{q}(1-q^{-s})^{-\Phi_{q}},$$ where the product runs over all prime powers $q.$ We denote by $Z_K(s)=-\sum\limits_{q}\frac{\Phi_{q}\log q}{q^{s}-1}$ its logarithmic derivative. One knows that $\zeta_K(s)$ can be analytically continued to the whole complex plane and satisfies a functional equation relating $\zeta_K(s)$ and $\zeta_K(1-s).$ Furthermore, in the function field case $\zeta_{K}(s)$ is a rational function of $t=r^{-s}$. Moreover, 
\begin{equation}
\label{rationality}
\zeta_{K}(s)=\frac{\prod\limits_{j=1}^{g}(\pi_j t - 1)(\bar{\pi}_j t - 1)}{(1-t)(1-rt)},
\end{equation}
and $|\pi_j|=\sqrt{r}$ (the Riemann hypothesis). For the rest of the paper we will assume that the Generalized Riemann Hypothesis is true for zeta functions of number fields, that is all the non-trivial zeroes of $\zeta_K(s)$ are on the line $\Re s = \frac{1}{2}.$

Here are our first main results:

\begin{theorem}[FF]
\label{limitff}
For a function field $K,$ an integer $N\geq 10$ and any $\epsilon=\epsilon_0+i\epsilon_1$ such that $\epsilon_0=\Re \epsilon > 0$ we have:
$$\sum_{f=1}^{N}\frac{f\Phi_{r^f}}{r^{(\frac{1}{2}+\epsilon)f}-1} + Z_{K}\left(\frac{1}{2}+\epsilon\right)+\frac{1}{r^{-\frac{1}{2}+\epsilon}-1} =O\left(\frac{g_K}{\epsilon_0 r^{\epsilon_0 N}} \right) + O\left(r^{\frac{N}{2}}\right).$$
\end{theorem}

\begin{theorem}[NF, GRH]
\label{limitnf}
For a number field $K,$ an integer $N\geq 10$ and any $\epsilon=\epsilon_0+i\epsilon_1$ such that $\epsilon_0=\Re \epsilon > 0$ we have:
\begin{multline*}
\sum_{q\leq N} \frac{\Phi_q\log q}{q^{\frac{1}{2}+\epsilon}-1}+ Z_{K}\left(\frac{1}{2}+\epsilon\right)+\frac{1}{{\epsilon-\frac{1}{2}}}= \\ 
= O\left(\frac{|\epsilon|^4+|\epsilon|}{\epsilon_0^2}(g+n\log N)\frac{\log^2 N}{N^{\epsilon_0}}\right) + O\left(\sqrt N\right).
\end{multline*}
\end{theorem}

Let us explain a little bit the meaning of these theorems. It was known before (see \cite{Zyk} and also below) that the identities (without the error terms) of the theorems are true in the asymptotic sense (when $N=\infty$ and $g=\infty$ for families of global fields). Our theorems give the "finite level" versions of these results. They allow to estimate how well the cutoffs of the series for $Z_{K}(s)$ approximate it away from the domain of convergence of this series (which is $\Re s > 1$) when we vary $K.$ 

We give the proof of these theorems in sections \ref{seclimitff} and \ref{seclimitnf} respectively. Both proofs are based on the Weil explicit formula. However, in the number field case the analytical difficulties are rather considerable, so the explicit formula has to be applied three times with different choices of test functions. We note that, as indicated in the remarks in the corresponding sections, in both cases we obtain the new proofs of the basic inequalities from \cite{Tsf92} and \cite{TV02}.

Our next results concern families of global fields $\{K_i\}$ with growing genus $g_i=g(K_i).$ Recall (\cite{TV97},\cite{TV02}) that a family of global fields is called asymptotically exact if the limits
$$\phi_{\alpha}=\phi_{\alpha}(\{K_i\})=\lim_{i\to\infty}\frac{\Phi_{\alpha}(K_i)}{g_i}$$
exist for each $\alpha$ which is a power of $r$ in the function field case and each prime power and $\alpha=\bbR$ and $\alpha=\bbC$ in the number field case. The numbers $\phi_{\alpha}$ are called the Tsfasman--Vl\u{a}du\c{t} invariants of the family $\{K_i\}.$ From now on we assume that all our families are asymptotically exact. 

We introduce the limit zeta function of a family $\{K_i\}$ as $$\zeta_{\{K_i\}}(s)=\prod\limits_{q} (1-q^{-s})^{-\phi_{q}}.$$ We will also denote by $Z_{\{K_i\}}(s)=-\sum\limits_{q}\frac{\phi_{q}\log q}{q^{s}-1}$ its logarithmic derivative. It follows from the basic inequality (cf. \cite{Tsf92} and \cite{TV02} or sections \ref{seclimitff} and \ref{seclimitnf} of this paper) that both the product and the sum converge absolutely for $\Re s \geq \frac{1}{2}$ and thus define analytic functions for $\Re s > \frac{1}{2}.$ 

Let us first formulate a corollary of theorems \ref{limitff} and \ref{limitnf}.

\begin{corollary}
\label{asscor}
For an asymptotically exact family of global fields $\{K_i\},$ an integer $N\geq 10$ and any $\epsilon=\epsilon_0+i\epsilon_1$ such that $\epsilon_0=\Re \epsilon > 0$ the following holds:
\begin{enumerate}
\item in the function field case:
$$\sum_{f=1}^{N}\frac{f\phi_{r^f}}{r^{(\frac{1}{2}+\epsilon)f}-1} + Z_{\{K_i\}}\left(\frac{1}{2}+\epsilon\right)=O\left(\frac{1}{\epsilon_0 r^{\epsilon_0 N}} \right);$$
\item in the number field case with the assumption of GRH:
$$\sum_{q\leq N}\frac{\phi_{q}\log q}{q^{\frac{1}{2}+\epsilon}-1} + Z_{\{K_i\}}\left(\frac{1}{2}+\epsilon\right)=O\left(\frac{(|\epsilon|^4+|\epsilon|)\log^3 N}{\epsilon_0^2 N^{\epsilon_0}}\right).$$
\end{enumerate}
\end{corollary}

This corollary, in particular, implies the convergence of the logarithmic derivatives of zeta functions of global fields to the logarithmic derivative of the limit zeta function for $\Re s > \frac{1}{2}$. This result (without an explicit error term but with a much easier proof) has been recently obtained in \cite{Zyk}.

Our next result concerns the behaviour of $Z_{\{K_i\}}(s)$ at $s = \frac{1}{2}.$

\begin{theorem}
\label{onehalf}
For an asymptotically exact family of global fields $\{K_i\}$ there exists a number $\delta>0$ depending on $\{K_i\}$ such that:
\begin{enumerate}
\item in the function field case:
$$\sum_{f=1}^{N}\frac{f\phi_{r^f}}{r^{\frac{f}{2}}-1}+Z_{\{K_i\}}\left(\frac{1}{2}\right)=O(r^{-\delta N});$$

\item in the number field case, assuming GRH, we have:
$$\sum_{q\leq N} \frac{\phi_q\log q}{\sqrt{q}-1}+Z_{\{K_i\}}\left(\frac{1}{2}\right)=O(N^{-\delta}).$$
\end{enumerate}
\end{theorem}

Let us formulate a corollary  of this result which, in a sense, improves the explicit Brauer--Siegel theorem from \cite{Leb}. We denote by $\varkappa_{K_i}=\Res\limits_{s=1}\zeta_{K_i}(s)$ the residue of $\zeta_{K_i}(s)$ at $s=1.$ We let $\kappa=\kappa_{\{K_i\}}=\lim\limits_{i\to\infty}\frac{\log\varkappa_{K_i}}{g_i}.$ One knows (\cite{TV97} and \cite{TV02}) that for an asymptotically exact family this limit exists and equals $\log \zeta_{\{K_i\}}(1)$ (we assume GRH in the number field case). In fact, in the number field case it can be seen as a generalization of the classical Brauer--Siegel theorem (cf. \cite{Lan94}).

\begin{corollary}
\label{BSExpl}
For an asymptotically exact family of global fields $\{K_i\}$ there exists a number $\delta>0$ depending on $\{K_i\}$ such that:
\begin{enumerate}
\item in the function field case:
$$\sum_{f=1}^{N}\phi_{r^f}\log\frac{r^f}{r^f-1}=\kappa+O\left(\frac{1}{r^{\left(\frac{1}{2}+\delta\right)N}N}\right);$$
\item assuming GRH, in the number field case:
$$\sum_{q\leq N}\phi_{q}\log\frac{q}{q-1}=\kappa+O\left(\frac{1}{N^{\frac{1}{2}+\delta}\log N}\right).$$
\end{enumerate}
\end{corollary}

We prove theorem \ref{onehalf} and both of the corollaries \ref{asscor} and \ref{BSExpl} in the section \ref{seconehalf}.

\subsection*{Acknowledgements} We would like to thank Jean-Fran\c{c}ois Burnol for his more than significant help with the estimates of the sums over zeroes of zeta functions. 

\section{Proof of theorem \ref{limitff}}
\label{seclimitff}

We will use the following analogue of Weil explicit formula for zeta functions of function fields, see \cite{Ser85} or \cite{LT} (in the case of varieties over finite fields) for a proof. 

\begin{theorem}
For a sequence ${v}=(v_n)$ such that $\sum\limits_{n=1}^{\infty} v_n r^{\frac{n}{2}}$ is convergent, the series $\sum\limits_{n=1}^{\infty} {v_n}r^{-\frac{n}{2}} \sum\limits_{m|n}m\Phi_{r^m}$ is also convergent and one has the following equality: 
\begin{equation*}
\sum_{n=1}^{\infty}v_n r^{-\frac{n}{2}} \sum_{f|n}f\Phi_{r^f}=\psi_v(r^{1/2})+\psi_v(r^{-1/2})-\sum_{j=1}^g\left(\psi_v\left(\frac{\pi_j}{\sqrt{r}}\right) +\psi_v\left(\frac{\bar{\pi}_j}{\sqrt{r}}\right)\right),
\end{equation*}
where the $\pi_j, \bar{\pi}_j$ are the inverse roots of the numerator of the zeta function of $K,$ $g=g_K$  and $\psi_v(t)=\sum\limits_{n=1}^{\infty}v_n t^n$.
\end{theorem}

Let us take the test sequence $v_n=v_n(N)=\frac{1}{r^{n\epsilon}}$ if $n\leq N$ and $0$ otherwise. Introducing it in the explicit formulae, we get $S_0(N,\epsilon)=S_1(N,\epsilon)+S_2(N,\epsilon)-S_3(N,\epsilon),$ where
\begin{align*}
S_0(N,\epsilon)=& \sum_{n=1}^{N}r^{-n\left(\frac{1}{2}+\epsilon\right)}\sum_{f|n}f\Phi_{r^f},\\
S_1(N,\epsilon)= & \sum_{n=1}^N r^{n\left(\frac{1}{2}-\epsilon\right)},\\
S_2(N,\epsilon)= & \sum_{n=1}^N r^{-n\left(\frac{1}{2}+\epsilon\right)},\\
S_3(N,\epsilon)= & \sum_{j=1}^g\sum_{n=1}^{N}r^{-n\left(\frac{1}{2}+\epsilon\right)}(\pi_j^n+\bar{\pi}_j^n).
\end{align*}
Let us estimate each of the $S_i.$

\subsection*{Calculation of $S_0$:}
Let us first change the summation order in $S_0$:
$$S_0(N,\epsilon)  =  \sum_{n=1}^{N}r^{-n\left(\frac{1}{2}+\epsilon\right)}\sum_{f|n}f\Phi_{r^f} =  \sum_{f=1}^{N}f\Phi_{r^f} \sum_{m=1}^{[ N/f]}\frac{1}{r^{fm\left(\frac{1}{2}+\epsilon\right)}}.$$
Now
\begin{eqnarray*}
R_0(N,\epsilon)&=&\sum_{f=1}^N f\Phi_{r^f} \frac{1}{r^{\left(\frac{1}{2}+\epsilon\right)f}-1} - S_0(N,\epsilon)\\
& = & \sum_{f=1}^{N}f \Phi_{r^f} \left(\frac{1}{r^{\left(\frac{1}{2}+\epsilon\right)f}-1}- \sum_{m=1}^{[ N/f]}r^{-fm\left(\frac{1}{2}+\epsilon\right)}\right) \\
& = & \sum_{f=1}^{N}f\Phi_{r^f}\sum_{m=[ N/f]+1}^\infty r^{-fm\left(\frac{1}{2}+\epsilon\right)}.
\end{eqnarray*}

Taking the absolute values, we can assume that $\epsilon$ is real. Summing the geometric series, we obtain 
$$0\leq \sum_{f=1}^N f\Phi_{r^f} \frac{1}{r^{\left(\frac{1}{2}+\epsilon\right)f}-1} -S_0(N,\epsilon)\leq \sum_{f=1}^{N}f\Phi_{r^f}r^{-\left(\frac{1}{2}+\epsilon\right)[ N/f ] f}\frac{1}{r^{\left(\frac{1}{2}+\epsilon\right)f} -1}.$$

We now use the Weil inequality $f\Phi_{r^f}\leq r^f+1+2g\sqrt{r^f}$, and split the above sum into two parts in the following way. For $f>[ N/2 ]$ we have $[ N/f ]=1$ and for $f\leq [ N/2 ]$ we use the inequality $f [ N/f ]\geq N-f.$ 

\begin{eqnarray*}
|R_0(N, \epsilon)|& \leq & \sum_{f=1}^{N} \frac{\left(1+r^f+2g \sqrt{r^f}\right)}{r^{f \left(\frac{1}{2}+\epsilon\right)[ N/f ]} \left(r^{\left(\frac{1}{2}+\epsilon\right)f}-1\right)}\\
& \leq & \sum_{f=1}^{[ N/2]}\frac{2\,r^{\left(\frac{1}{2}-\epsilon\right)f}+4g\,r^{-f\epsilon}}{r^{(N-f)\left(\frac{1}{2}+\epsilon\right)}} +\sum_{f>[ N/2]}^N \frac{2\,r^{\left(\frac{1}{2}-\epsilon\right)f}+4g\,r^{-f\epsilon}}{r^{f\left(\frac{1}{2}+\epsilon\right)}} \\
& \leq & \frac{2}{r^{N\left(\frac{1}{2}+\epsilon\right)}}\sum_{f=1}^{[ N/2 ]} (r^{f}+2\,g\,r^{\frac{f}{2}})+2\sum_{f>[ N/2]} (r^{-2\epsilon f}+2\,g\,r^{-\left(\frac{1}{2}+ 2\epsilon\right)f})\\
& \leq & \frac{2}{r^{N\left(\frac{1}{2}+\epsilon\right)}} \left (\frac{r^{\frac{N}{2}+1}-r}{r-1}+2g\,\frac{r^{\frac{N}{4}+\frac{1}{2}}-r^{\frac{1}{2}}} {r^{\frac{1}{2}}-1} \right)  
+ \frac{2r^{-\epsilon N}}{1-r^{-2\epsilon}}+\frac{4g\, r^{-\frac{N}{4}-\epsilon N}}{1-r^{-\frac{1}{2}-2\epsilon}}\\
& \leq & \frac{16}{r^{\epsilon N}} \left ( 2g\,r^{-\frac{N}{4}}+\frac{1}{r^{\epsilon}-1}\right) \leq  \frac{32}{r^{\epsilon N}} \left ( g\,r^{-\frac{N}{4}}+\frac{1}{\epsilon}\right).
\end{eqnarray*}

\subsection*{Calculation of $S_1$:}

$$0\leq |S_1(N,\epsilon)|\leq r^{\frac{1}{2}-\epsilon_0}\cdot\frac{r^{\left(\frac{1}{2}-\epsilon_0\right)N}-1}{r^{\frac{1}{2}-\epsilon_0}-1}\ll r^{\frac{N}{2}}.$$

\subsection*{Calculation of $S_2$:}

$$0\leq |S_2(N,\epsilon)|\leq \frac{1-r^{-\left(\frac{1}{2}+\epsilon_0\right)N}}{r^{\frac{1}{2}+\epsilon_0}-1}\leq 4.$$

\subsection*{Calculation of $S_3$:}
$$R_3(N, \epsilon)=S_3(N,\epsilon)-\sum_{j=1}^g \left(\frac{\pi_j}{r^{\frac{1}{2}+\epsilon}-\pi_j} +\frac{\bar{\pi}_j}{r^{\frac{1}{2}+\epsilon}-\bar{\pi}_j}\right)=-\sum_{j=1}^g\sum_{n=N+1}^\infty \left(\frac{\pi_j}{r^{\frac{1}{2}+\epsilon}}\right)^n+\left(\frac{\bar{\pi}_j}{r^{\frac{1}{2}+\epsilon}}\right)^n.$$
The absolute value of the right hand side can be bounded using the fact that $|\pi_j|\leq r^{\frac{1}{2}}$:
$$|R_3(N,\epsilon)|=\left|\sum_{j=1}^g\sum_{n=N+1}^\infty \left(\frac{\pi_j}{r^{\frac{1}{2}+\epsilon}}\right)^n+\left(\frac{\bar{\pi}_j}{r^{\frac{1}{2}+\epsilon}}\right)^n\right|\leq 2g\, \frac{r^{-N\epsilon_0}}{r^{\epsilon_0}-1}\leq 4g\, \frac{r^{-N\epsilon_0}}{\epsilon_0}.$$

From the expression (\ref{rationality}) of $\zeta_K(s)$ as rational function in $t=r^{-s}$ we can easily deduce the following formula for its logarithmic derivative:
$$Z_K\left(\frac{1}{2}+\epsilon\right)=-\frac{1}{r^{\frac{1}{2}+\epsilon}-1}-\frac{1}{r^{-\frac{1}{2}+\epsilon}-1} + \sum_{j=1}^g\left( \frac{\pi_j}{r^{\frac{1}{2}+\epsilon}-\pi_j} +\frac{\bar{\pi}_j}{r^{\frac{1}{2}+\epsilon}-\bar{\pi}_j}\right).$$

Putting it all together we get the statement of the theorem. \qed

\begin{remark}
Using our theorem we can easily reprove the basic inequality from \cite{TV97}. We take a real $\epsilon < \frac{1}{4},$ and remark that
$$Z_{K}\left(\frac{1}{2}+\epsilon\right)+\frac{1}{r^{\frac{1}{2}+\epsilon}-1}+\frac{1}{r^{-\frac{1}{2}+\epsilon}-1}+g=\sum_{j=1}^g\left( \frac{\pi_j}{r^{\frac{1}{2}+\epsilon}-\pi_j} +\frac{\bar{\pi}_j}{r^{\frac{1}{2}+\epsilon}-\bar{\pi}_j}+1\right)\geq 0,$$ 
as
$$\frac{\pi_j}{r^{\frac{1}{2}+\epsilon}-\pi_j} +\frac{\bar{\pi}_j}{r^{\frac{1}{2}+\epsilon}-\bar{\pi}_j}+1=\frac{r^{1+2\epsilon}-|\pi_j|^2}{(r^{\frac{1}{2}+\epsilon}-\pi_j)(r^{\frac{1}{2}+\epsilon}-\bar{\pi}_j)}\geq 0.$$ 
Now, from the theorem we get that
$$\sum_{f=1}^{N}\frac{f\Phi_{r^f}}{r^{(\frac{1}{2}+\epsilon)f}-1}\leq g + O\left(\frac{g}{\epsilon r^{\epsilon N}}\right) + O(r^{\frac{N}{2}}).$$
We divide by $g$ and first let $g\to\infty$ (varying $K$), after that we let $N\to\infty$ and finally we take the limit when $\epsilon\to 0.$ In doing so we obtain the basic inequality from \cite{Tsf92}:
$$\sum_{f=1}^\infty \frac{f\phi_{r^f}}{r^{\frac{f}{2}}-1}\leq 1.$$
\end{remark}

\section{Proof of theorem \ref{limitnf}}
\label{seclimitnf}

Our starting point will be the Weil explicit formula, the proof of which can be found in \cite{Poi} or in \cite[chap. XVII]{Lan94} (with slightly more general conditions on the test functions).

Consider the class $(W)$ of even real valued functions, satisfying the following conditions:
\begin{enumerate}
\item there exists $\epsilon>0$ such that $\int_{0}^\infty F(x)e^{(\frac{1}{2}+\epsilon)x}\, dx$ is convergent in the sense of Cauchy;
\item there exists $\epsilon>0$ such that $ F(x)e^{(\frac{1}{2}+\epsilon)x}$ has bounded variation; 
\item $\frac{F(0)-F(x)}{x}$ has bounded variation;
\item for any $x$ we have $F(x)=\frac{F(x-0)+F(x+0)}{2}.$
\end{enumerate}

For such a function $F$ we define
\begin{equation}
\label{Mellin}
\phi(s)=\int_{-\infty}^{+\infty}F(x)e^{(s-\frac{1}{2})x}\,dx.
\end{equation}

The Weil explicit formula for Dedekind zeta functions of number fields can be stated as follows:

\begin{theorem}[Weil]
\label{WeilTh}
Let $K$ be a number field. Let $F$ belong to the class $(W)$ and let $\phi(s)$ be defined by $(\ref{Mellin}).$ Then the sum $\sum\limits_{|\Im\rho|<T}\phi(\rho),$ where $\rho$ runs through the non-trivial zeroes of the Dedekind zeta-function of $K,$ is convergent when $T\to\infty$ and the limit $\sum\limits_\rho\phi(\rho)$ is given by:

\begin{multline}
\label{Weil}
\sum_\rho \phi(\rho)= F(0)\left(2g-n(\gamma+\log 8\pi)-r_1\frac{\pi}{2}\right)+4\int_0^\infty F(x)\ch\left(\frac{x}{2}\right)\\
+ r_1\int_0^\infty \frac{F(0)-F(x)}{2\ch(\frac{x}{2})}\,dx+{n}\int_0^\infty\frac{F(0)-F(x)}{2\sh(\frac{x}{2})}\,dx -2\sum_{\pg,m}\frac{\log\rmN\pg}{\rmN\pg^\frac{m}{2}}F(m\log\rmN\pg),
\end{multline}
where the last sum is taken over all prime ideals $\pg$ in $K$ and all integers $m\geq 1.$
\end{theorem}

First of all, we remark that, if we have a complex valued function $F(x)$ with both real and imaginary parts $F_0(x)$ and $F_1(x)$ being even and lying in $(W),$ we can apply (\ref{Weil}) separately to $F_0(x)$ and $F_1(x).$ Thus the explicit formula, being linear in the test function, is also applicable to the initial complex valued function $F(x).$  

We apply the explicit formula to the function defined by
$$F_{N,\epsilon}(x)=\begin{cases}e^{-\epsilon|x|} & \text{if } |x|<\log (N+\frac{1}{2}),\\
0 & \text{if } |x|>\log (N+\frac{1}{2})
\end{cases}$$
(here $N+\frac{1}{2}$ is take to avoid counting some of the terms with the factor $\frac{1}{2}$).

Next, we estimate each of the terms in (\ref{Weil}).

\subsection{The sum over the primes.} 

\begin{align*}
\sum_{\pg,m}\frac{\log\rmN\pg}{\rmN\pg^\frac{m}{2}}F_{N,\epsilon}(m\log\rmN\pg) & =\sum_{\rmN\pg^m\leq N}\frac{\log{\rmN\pg}}{\rmN\pg^{(\frac{1}{2}+\epsilon)m}}\\
& =\sum_{{\rmN\pg}\leq N}\frac{\log{\rmN\pg}}{\rmN\pg^{\frac{1}{2}+\epsilon}-1} 
-\sum_{{\rmN\pg}\leq N}\log{\rmN\pg}\sum_{m>\frac{\log{N}}{\log\rmN\pg}}\frac{1}{\rmN\pg^{(\frac{1}{2}+\epsilon)m}}.
\end{align*}

We have to estimate the sum:
$$\Delta(N,\epsilon)=\sum_{{\rmN\pg}\leq N}\log{\rmN\pg}\sum_{m>\frac{\log{N}}{\log\rmN\pg}}\frac{1}{\rmN\pg^{(\frac{1}{2}+\epsilon)m}}.$$

Taking the absolute values, we can assume that $\epsilon$ is real. Calculating the remainder term of the geometric series, we get:
$$\Delta(N,\epsilon)\leq (2+\sqrt{2})\sum_{{\rmN\pg}\leq N}\frac{\log{\rmN\pg}}{\rmN\pg^{(\frac{1}{2}+\epsilon)([\frac{\log{N}}{\log\rmN\pg}]+1)}}$$
(for $(1-\rmN\pg^{-1/2-\epsilon})^{-1}\leq (1-2^{-1/2})^{-1}\leq \sqrt{2}(1+\sqrt{2})$).

Let us split the sum into two parts according as whether $\rmN\pg>\sqrt{N}$ or not. Taking into account that $\log\N\pg[\log N/\log\N\pg]\geq \log N-\log \N\pg$ for $\log\N\pg\leq [\log\sqrt{\N\pg}],$ we obtain:
$$\Delta(N,\epsilon)\leq (2+\sqrt{2})\left(\sum_{{\rmN\pg}\leq \sqrt{N}}\frac{\log{\rmN\pg}}{e^{\log\rmN{(\frac{1}{2}+\epsilon)}}}+\sum_{\sqrt{N}<\rmN\pg\leq N}\frac{\log{\rmN\pg}}{\rmN\pg^{(1+2\epsilon)}}\right).$$
Write $$\Delta_1(N,\epsilon)=\sum\limits_{{\rmN\pg}\leq \sqrt{N}}\frac{\log{\rmN\pg}}{e^{\log\rmN{(\frac{1}{2}+\epsilon)}}},$$  $$\Delta_2(N,\epsilon)=\sum\limits_{\sqrt{N}<\rmN\pg\leq N}\frac{\log{\rmN\pg}}{\rmN\pg^{(1+2\epsilon)}}.
$$ 
For $\Delta_1(N,\epsilon)$ we have:
$$\Delta_1(N,\epsilon)\leq \frac{1}{N^{\frac{1}{2}+\epsilon}}\sum\limits_{{\rmN\pg}\leq \sqrt{N}}\log{\rmN\pg}.$$
The last sum can be estimated with the help of Lagarias and Odlyzko results (which use GRH, cf. \cite[Theorem 9.1]{LO}):
$$\sum_{{\rmN\pg}\leq \sqrt{N}}\log{\rmN\pg}\leq\sum_{{\rmN\pg}^k\leq \sqrt{N}}\log{\rmN\pg}= \sqrt{N}+ O(N^{\frac{1}{4}}\log{N}(g+n\log{N}))$$ 
with an effectively computable absolute constant in $O$. Thus we get:
$$\Delta_1(N,\epsilon)\leq \frac{2+\sqrt{2}}{N^{\epsilon}}+ a_0\frac{g\log N+n\log^2{N}}{N^{\frac{1}{4}+\epsilon}}. $$

We can estimate the sum $\Delta_2(N,\epsilon)$ as follows:
$$\Delta_2(N,\epsilon)\leq\int_{\sqrt{N}}^\infty \frac{\log t}{t^{1+2\epsilon}}\,d\pi(t),$$ where $\pi(t)$ is the prime counting function $\pi(t)=\sum\limits_{\rmN\pg\leq t} 1$. As before, according to Lagarias and Odlyzko, $\pi(t)=\int_2^t \frac{\,dx}{\log{x}}+ \delta(t),$ with $|\delta(t)|\leq a_1 \sqrt{t}(g+n\log{t}).$ Thus, substituting, we get:  
$$\Delta_2(N,\epsilon)\leq\int_{\sqrt{N}}^\infty t^{-1-2\epsilon}\,dt+2 |\delta(\sqrt{N})|\frac{\log N}{N^{\frac{1}{2}+\epsilon}}+\left|\int_{\sqrt{N}}^\infty\delta(t)\frac{1-(1+2\epsilon)\log{t}}{t^{2+2\epsilon}}\,dt\right|.$$
We deduce that
$$\Delta_2(N,\epsilon)\leq \frac{1}{2\epsilon N^\epsilon}+2 a_1 (g+n\log{N})
\frac{\log N}{N^{\frac{1}{4}+\epsilon}}+\int_{\sqrt{N}}^\infty  a_1(g+n\log{t}) \frac{|1-(1+2\epsilon)\log{t}|}{t^{\frac{3}{2}+2\epsilon}}\,dt.$$

For $N\geq 8$ we have: 
$$\int_{\sqrt{N}}^\infty  a_1(g+n\log{t}) \frac{|1-(1+2\epsilon)\log{t}|}{t^{\frac{3}{2}+2\epsilon}}\,dt\leq \int_{\sqrt{N}}^\infty  a_1(g+n\log{t}) \frac{(1+2\epsilon)\log{t}}{t^{\frac{3}{2}+2\epsilon}}\,dt.$$
Integrating by parts, we can find that
$$\int_{\sqrt{N}}^\infty \frac{\log{t}}{t^{\frac{3}{2}+2\epsilon}}\,dt=\frac{\log{N}}{2(\frac{1}{2}+2\epsilon)N^{\frac{1}{4}+\epsilon}}+\frac{1}{(\frac{1}{2}+2\epsilon)^2 N^{\frac{1}{4}+\epsilon}},$$
and  $$\int_{\sqrt{N}}^\infty \frac{\log^2{t}}{t^{\frac{3}{2}+2\epsilon}}\,dt=\frac{\log^2{N}}{4(\frac{1}{2}+2\epsilon)N^{\frac{1}{4}+\epsilon}}+\frac{\log{N}}{2(\frac{1}{2}+2\epsilon)^2 N^{\frac{1}{4}+\epsilon}}+\frac{1}{(\frac{1}{2}+2\epsilon)^3 N^{\frac{1}{4}+\epsilon}}.$$

We conclude that the following estimate holds: 
$$\Delta_2(N,\epsilon)\leq \frac{1}{2\epsilon N^\epsilon}+ a_2 \left(\frac{n \log^2 N}{N^{\frac{1}{4}+\epsilon}} +  \frac{g \log N}{N^{\frac{1}{4}+\epsilon}}\right).$$

Putting everything together, we see that:
\begin{equation}
\label{primesest}
|\Delta(N,\epsilon)|\ll \frac{1}{\epsilon_0 N^{\epsilon_0}}+ \frac{\log N}{N^{\frac{1}{4}+\epsilon_0}} (n \log N +  g ).
\end{equation}

\subsection{Archimedean terms}

First of all, 

\begin{equation}
\label{arch1}
\left|\int_0^\infty F_{N,\epsilon}(x) \ch\left(\frac{x}{2}\right)\,dx \right| \leq \int_0^{\log(N+\frac{1}{2})} e^{(\frac{1}{2}-\epsilon_0) x}\,dx =  \frac{(N+\frac{1}{2})^{\frac{1}{2}-\epsilon_0}-1}{\frac{1}{2}-\epsilon_0}\ll \sqrt N.
\end{equation}

Let 
$$I_{N,\epsilon}=\int_0^\infty \frac{1-F_{N,\epsilon}(x)}{2\sh(\frac{x}{2})}\, dx$$ 
and 
$$I_{\infty,\epsilon}=\int_0^\infty \frac{1-e^{-\epsilon x}}{2\sh(\frac{x}{2})}\,dx.$$

We have for $N\geq 4:$
$$|I_{\infty,\epsilon}-I_{N,\epsilon}|\leq \int_{\log{N}}^\infty \frac{2}{e^{\frac{x}{2}}}\,dx\leq \frac{4}{\sqrt{N}}.$$

Now, 
\begin{align*}
I_{\infty,\epsilon}&= \int_0^\infty \left(\frac{e^{-\frac{x}{2}}}{1-e^{-x}} -\frac{e^{-(\frac{1}{2}+\epsilon) x}}{1-e^{-x}}\right)\,dx\\
&= \int_0^\infty \left(\left(\frac{e^{-\frac{x}{2}}}{1-e^{-x}}-\frac{e^{-x}}{x}\right)+ \left(\frac{e^{-x}}{x}-\frac{e^{-(\frac{1}{2}+\epsilon) x}}{1-e^{-x}}\right)\right)\,dx\\
&=\psi\left(\frac{1}{2}+\epsilon\right)-\psi\left(\frac{1}{2}\right),
\end{align*}
as $$\psi(x)=\frac{\Gamma'(x)}{\Gamma(x)}=\int_0^\infty \left(\frac{e^{-t}}{x}-\frac{e^{-x{t}}}{1-e^{-t}}\right)\,dt.$$

The second integral
$$J_{N,\epsilon}=\int_0^\infty \frac{1-F_{N,\epsilon}(x)}{2\ch(\frac{x}{2})}\, dx$$ 

can be estimated along the same lines using an integral from \cite[3.541]{GR} :
$$\int_0^\infty \frac{e^{-\epsilon x}}{\ch(\frac{x}{2})}\,dx=\psi\left(\frac{1}{4}+\frac{\epsilon}{2}\right)-\psi\left(\frac{3}{4}+\frac{\epsilon}{2}\right).$$

Taking into account that $\psi(2x)=\frac{1}{2}\left(\psi(x)+\psi\left(x+\frac{1}{2}\right)\right)+\log 2,$ we finally obtain:
\begin{equation}
\label{arch2}
\begin{split}
J_{N,\epsilon}&= \frac{\pi}{2} + \log 2+ \psi\left(\frac{1}{4}+\frac{\epsilon}{2}\right)-\psi\left(\frac{1}{2}+\epsilon\right)+O\left(\frac{1}{\sqrt{N}}\right),\\
I_{N,\epsilon}&= \gamma+\log 4 + \psi\left(\frac{1}{2}+\epsilon\right)+O\left(\frac{1}{\sqrt{N}}\right).
\end{split}
\end{equation}

\subsection{The sum over the zeroes: the main term}

Let us estimate now the sum $\sum\limits_\rho\phi(\rho)$ over zeroes of $\zeta_K(s).$ Let $\rho=\frac{1}{2}+it$ be a zero of the zeta function of $K$ on the critical line. Put $y=\log(N+\frac{1}{2}).$ We have 
$$\phi(\rho)=\int_{-y}^{y}e^{-\epsilon|x|+it x} \,dx=\int_0^y e^{(-\epsilon+it)x}\,dx+\int_0^y e^{(-\epsilon-it)x}\,dx,$$ 
so  
$$\phi(\rho)=\frac{2}{\epsilon^2+t^2}(\epsilon+e^{-\epsilon y}(-\epsilon \cos(ty)+t \sin(ty))).$$

We divide the sum over $\rho$ into three parts: 
\begin{align*}
S_1(\epsilon)&=\sum\limits_{\rho=\frac{1}{2}+it} \frac{\epsilon}{\epsilon^2+t^2};\\
S_2(y, \epsilon)&=\sum\limits_{\rho=\frac{1}{2}+it} \frac{\cos(ty)}{\epsilon^2+t^2};\\
S_3(y, \epsilon)&=\sum\limits_{\rho=\frac{1}{2}+it} \frac{t\sin(ty)}{\epsilon^2+t^2};
\end{align*}
so that
$$\sum\limits_\rho\phi(\rho)=2S_1(\epsilon)-2\epsilon e^{-\epsilon y}S_2(y, \epsilon)+2e^{-\epsilon y}S_3(y,\epsilon).$$

Let us relate the sum $S_1(\epsilon)$ to $Z_K(s),$ the logarithmic derivative of $\zeta_K(s)$. Stark's formula (cf. \cite[(9)]{Sta}) gives us the following:
\begin{equation}
\label{StarkFormula}
\sum_\rho \frac{1}{s-\rho}= \frac{1}{s-1}+\frac{1}{s}+g-\frac{n}{2}\log\pi+\frac{r_1}{2}\psi\left(\frac{s}{2}\right) +r_2(\psi(s)-\log 2)+Z_K(s),
\end{equation}
where as before $\psi(s)=\frac{\Gamma'(s)}{\Gamma(s)}.$ Specializing at $s=\frac{1}{2}+\epsilon,$ we obtain:

\begin{multline}
\label{zeromain}
\sum_{\rho=\frac{1}{2}+it} \frac{\epsilon}{\epsilon^2+t^2}= \frac{1}{\epsilon-\frac{1}{2}}+\frac{1}{\epsilon+\frac{1}{2}}+g-\frac{n}{2}\log\pi - r_2\log 2\\ 
+\frac{r_1}{2}\psi\left(\frac{1}{4}+\frac{\epsilon}{2}\right) + r_2\psi\left(\frac{1}{2}+\epsilon\right)+Z_K\left(\frac{1}{2}+\epsilon\right).
\end{multline}

We note that the archimedean factors from the Stark formula and from the initial Weil explicit formula cancel each other. We are left to prove that $S_2(y, \epsilon)$ and $S_3(y, \epsilon)$ are sufficiently small.

\subsection{The sum over the zeroes: the remainder term.}
To estimate 
$$S_2(y, \epsilon)=\sum\limits_{\rho=\frac{1}{2}+it} \frac{\cos(ty)}{\epsilon^2+t^2}$$ 
we take the absolute values of all the terms in the sum so that
\begin{equation}
\label{zeroest}
|S_2(y, \epsilon)|\leq \sum_{\rho=\frac{1}{2}+it}\frac{1}{|\epsilon^2+t^2|}\leq \sum_{\rho=\frac{1}{2}+it}\frac{n(j)}{\epsilon_0^2+(t-|\epsilon_1|)^2},
\end{equation}
where $n(j)$ is the number of zeroes with $|t-j|<1.$ A standard estimate from \cite[Lemma 5.4]{LO} yields $n(j) \ll g+n\log(j+2),$ thus
\begin{align*}
|S_2(y, \epsilon)| &\ll \frac{g+n\log(|\epsilon_1|+2)}{\epsilon_0^2}+ g+n\sum_{j=1}^{|\epsilon_1|+1}\frac{\log j}{|\epsilon_1|+2-j}+ g+n\log(|\epsilon_1|+2)\\
&\ll (g+n\log^2(|\epsilon_1|+2))\left(1+\frac{1}{\epsilon_0^2}\right).
\end{align*}

Let us finally estimate 
$$S_3(y,\epsilon)=\sum_{\rho=\frac{1}{2}+it}\frac{t\sin(ty)}{\epsilon^2+t^2}.$$ 

We have 
$$S_3(y,\epsilon)=\sum_{\rho=\frac{1}{2}+it}\frac{\sin{ty}}{t}-\sum_{\rho=\frac{1}{2}+it}\frac{\epsilon^2\sin(ty)}{t(\epsilon^2+t^2)}=A(y)-B(y,\epsilon).$$ 

The series for the formal derivative of $B(y,\epsilon)$ with respect to $y$ is given by 
$$\sum_{\rho=\frac{1}{2}+it}\frac{\epsilon^2\cos(ty)}{\epsilon^2+t^2}.$$

Using the estimates for $S_2(y, \epsilon)$ we deduce that on any compact subset of $[0, +\infty)$ this series is absolutely and uniformly convergent to $B'(y)$, and we have $|B'(y,\epsilon)|\ll |\epsilon|^2(g+n\log^2(|\epsilon_1|+2))\left(1+\frac{1}{\epsilon_0^2}\right).$ Thus we see that $|B(y)| \ll y|\epsilon|^2(g+n\log^2(|\epsilon_1|+2))\left(1+\frac{1}{\epsilon_0^2}\right),$ since $B(0,\epsilon)=0.$

\subsection{The sum over the zeroes: the difficult part.}

We are left to estimate the term $A(y).$ 

Let us recall a particular case of Weil explicit formula which is due to Landau (cf. \cite{Lan71}):
\begin{equation}
\label{Landau}
\sum\limits_\rho \frac{x^\rho}{\rho}=x-\Psi(x)-r\log{x}-b-\frac{r_1}{2}\log(1-x^{-2})-r_2\log(1-x^{-1}),
\end{equation}
where $\Psi(x)=\sum\limits_{\rmN\pg^k\leq x}\log\rmN\pg,$ $b$ is the constant term of the expansion of $Z_K(s)$ at $0,$ $r= r_1+r_2-1$ and $x$ is not a prime power. This formula is  stated in \cite{Lan71} for $x\geq \frac{3}{2},$ however, applying theorem \ref{WeilTh} to the function
$$F_{x}(y)=\begin{cases}e^{|y|/2} & \text{if } |y|<\log x,\\
0 & \text{if } |y|> \log x,
\end{cases}$$
one can see that it is valid for any $x > 1.$ We also note that by an effective version of the prime ideals theorem (\cite[Theorem 9.1]{LO}) we have the following estimate:
\begin{equation}
\label{LagOdl}
\Psi(x)-x=O\left(x^{\frac{1}{2}}\log{x}(g+n\log{x})\right).
\end{equation}

Now, we introduce $C(x)=\sum\limits_\rho \frac{x^\rho}{\rho},$ $D(x)=\sum\limits_{\rho\neq\frac{1}{2}} \frac{x^\rho}{\rho-\frac{1}{2}}$ and $E(x)=D(x)-C(x).$ From (\ref{Landau}) and (\ref{LagOdl}) we see that $C(x)$ is an integrable function on compact subsets of $\left(1, +\infty\right).$ Using the arguments similar to those from the previous subsection we can deduce that the series for $E(x)$ is absolutely and uniformly convergent on compact subsets of $\left[1, +\infty\right)$ and thus $E(x)$ is a continuous function on this interval. From this we conclude that the series for $D(x)$ is also convergent to a locally integrable function.

If we put $x=e^{y},$ we get 
$$\Re D(e^y)=e^{\frac{y}{2}}\sum_{\rho\neq\frac{1}{2}}\frac{\sin(ty)}{t},$$ 
which is equal to $e^{\frac{y}{2}}A(y)$ up to a term corresponding to a possible zero of $\zeta_K(s)$ at $\rho =\frac{1}{2}.$

Since the series for $C(x)$ is not uniformly convergent, we will have to work with distributions defined by $C(x), D(x)$ and $E(x).$ See \cite{Sch} for the basic notions and results used here. From the fact that a convergent series of distributions can be differentiated term by term we deduce that the following equality holds:
$$\frac{d}{dx}\frac{E(x)}{\sqrt{x}}=\frac{C(x)}{2\sqrt{x^{3}}}.$$ 

We apply (\ref{Landau}) to the right hand side of this formula and integrate from $1+\delta$ to $x$ (here $\delta >0$). The obtained equality will be valid in the sense of distributions, thus almost everywhere for the corresponding locally integrable functions defining these distributions. Since $E(x)$ is continuous, we see that the resulting identity

\begin{multline*}
\frac{E(x)}{\sqrt{x}}= E(1+\delta)+\int_{1+\delta}^x \frac{t-\Psi(t)}{2t^{\frac{3}{2}}}\,dt -r\int_{1+\delta}^x \frac{\log{t}}{2t^{\frac{3}{2}}}\,dt\\
 -\int_{1+\delta}^x \frac{b}{2t^{\frac{3}{2}}}\,dt-\frac{r_1}{2}\int_{1+\delta}^x \frac{\log{(1-t^{-2})}}{2t^{\frac{3}{2}}}\,dt-r_2\int_{1+\delta}^x \frac{\log{(1-t^{-1})}}{2t^{\frac{3}{2}}}\,dt
\end{multline*}
actually holds pointwise on $\left[1+\delta, +\infty\right).$ We use (\ref{LagOdl}) to estimate $t-\Psi(t).$ It is easily seen that all the integrals converge when $\delta\to 0.$ From \cite[10.RH]{Lan71} it follows that $b\ll g+n$. 
$$E(1)=\sum\limits_{\rho\neq \frac{1}{2}} \frac{1}{\rho-\frac{1}{2}}-\sum\limits_{\rho}\frac{1}{\rho}=-\frac{1}{2}\sum\limits_{\rho=\frac{1}{2}+it}\frac{1}{\frac{1}{4}+t^2},$$ 
the first sum being zero as the term in $\rho$ and $1-\rho$ cancel each other. An estimate for the last sum can be made using (\ref{zeroest}). This gives $|E(1)|\ll g+n.$  Putting it all together we see that $|E(x)|\ll \sqrt{x}\log^2 x (g+n\log x).$ The estimate $|C(x)|\ll \sqrt{x}\log^2 x (n+g)$ can be obtained directly using (\ref{LagOdl}). Thus, we conclude that $|A(y)|\ll y^2(g+ny).$

Finally, combining all together we get:
\begin{equation*}
\sum_\rho \phi(\rho) = 2 S_1(\epsilon)+O\left(\frac{|\epsilon|^4+|\epsilon|}{\epsilon_0^2}(g+n\log N)\frac{\log^2 N}{N^{\epsilon_0}}\right).
\end{equation*}

This estimate together with (\ref{primesest}), (\ref{arch1}), (\ref{arch2}) and (\ref{zeromain}) completes the proof of the theorem. \qed

\begin{remark}

Using our theorem we can derive the basic inequality from \cite{TV02}. Indeed, we apply the formula (\ref{zeromain}) to express $Z_K\left(\frac{1}{2}+\epsilon\right)$ via $\sum\limits_{\rho=\frac{1}{2}+it} \frac{\epsilon}{\epsilon^2+t^2}$ plus some archimedean terms. For a real positive $\epsilon < \frac{1}{4}$ the latter sum is non-negative, thus we see that

\begin{multline*}
\sum_{q\leq N} \frac{\Phi_q\log q}{q^{\frac{1}{2}+\epsilon}-1}+ \frac{n}{2}\log\pi + r_2\log 2 -\frac{r_1}{2}\psi\left(\frac{1}{4}+\frac{\epsilon}{2}\right) - r_2\psi\left(\frac{1}{2}+\epsilon\right) \\ 
\leq g+ O\left((g+n\log N)\frac{\log^2 N}{\epsilon N^\epsilon}\right) + O\left(\sqrt N\right).
\end{multline*}

Now, we divide by $g$ and first let $g\to\infty$ (varying $K$), after that we let $N\to\infty$ and finally we take the limit when $\epsilon\to 0.$ Taking into account that $\psi(\frac{1}{2})=-\gamma-2\log 2$ and $\psi(\frac{1}{4})=-\frac{\pi}{2}-\gamma-3\log 2,$ we obtain the basic inequality from \cite{Tsf92}:
$$\sum_{q}\frac{\phi_{q}\log q}{\sqrt{q}-1}+\phi_{\bbR}\left(\log (2\sqrt{2\pi})+\frac{\pi}{4}+\frac{\gamma}{2}\right)+\phi_{\bbC}\left(\log(8\pi)+\gamma \right)\leq 1.$$
\end{remark}

\begin{remark}
The choice of the test function $F_{N,\epsilon}(x)$ in the explicit formula is not accidental. Indeed, the resulting formulas "approximate" the Stark formula (\ref{StarkFormula}) when $N\to \infty.$
\end{remark}

\section{Proof of theorem \ref{onehalf} and of the corollaries}
\label{seconehalf}

We will carry out the proofs in the function field case, the calculations in the number field case being exactly the same.

\begin{proof}[Proof of the corollary \ref{asscor}:]
Assume first that $\epsilon \neq \frac{1}{2}+\frac{2\pi i k}{\log r},$ $k\in \bbZ.$ We note that 
\begin{multline*}
\left|\sum_{f=1}^{\infty}\frac{f\phi_{r^f}}{r^{(\frac{1}{2}+\epsilon)f}-1} + \frac{1}{g_j}Z_{K_j}\left(\frac{1}{2}+\epsilon\right)\right| \leq \\ \leq \left|\sum_{f=N+1}^{\infty}\frac{f\phi_{r^f}}{r^{(\frac{1}{2}+\epsilon)f}-1}\right|  
+\sum_{f=1}^{N}\frac{f \left|\frac{\Phi_{r^f}}{g_j}-\phi_{r^f}\right|}{r^{(\frac{1}{2}+\epsilon)f}-1}+ \frac{1}{g_j}\left|\sum_{f=1}^{N}\frac{f\Phi_{r^f}}{r^{(\frac{1}{2}+\epsilon)f}-1}+Z_{K_j}\left(\frac{1}{2}+\epsilon\right)\right|.
\end{multline*}

Given $\delta >0$ we choose an integer $N$ such that the first sum is less than $\delta$ (this is possible due to the basic inequality) and such that $\epsilon_0 r^{\epsilon_0 N} \geq \frac{1}{\delta}.$ Now, taking $g$ sufficiently large, and using theorem \ref{limitff} as well as the convergence of $\frac{\Phi_{r^f}}{g_j}$ to $\phi_{r^f},$ we conclude that the whole sum is $\ll \delta.$ Thus, we deduce that
\begin{equation}
\label{limitformula}
\lim_{j\to \infty}\frac{Z_{K_j}\left(\frac{1}{2}+\epsilon\right)}{g_j}=Z_{\{K_j\}}\left(\frac{1}{2}+\epsilon\right).
\end{equation}

Now, the corollary immediately follows from theorem \ref{limitff} and (\ref{limitformula}). Though we initially assumed that $\epsilon\neq \frac{1}{2}+\frac{2\pi i k}{\log r},$ the statement still holds for $\epsilon=\frac{1}{2}+\frac{2\pi i k}{\log r}$ as all the function are continuous (and even analytic) for $\Re \epsilon > 0.$
\end{proof}

\begin{remark}
The formula (\ref{limitformula}) no longer holds when $\epsilon=0$ as can be seen from the fact that $Z_K\left(\frac{1}{2}\right)=g_K-1.$ In fact, the identity holds if and only if our family is asymptotically optimal. Whether it holds or not for the logarithm of $\zeta_K(s)$ and not for its derivative seems to be very difficult to say at the moment. Even for quadratic fields this question is far from being obvious. It is known that in the number field case there exists a sequence $(d_i)$ in $\bbN$ of density at least $\frac{1}{2}$  such that
$$\lim\limits_{i\to\infty}\frac{\log \zeta_{\bbQ(\sqrt{d_i})}(\frac{1}{2})}{\log d_i}=0$$
(cf. \cite{IS}). The techniques of the evaluation of mollified moments of Dirichlet $L$- functions used in that paper is rather involved. In general one can prove an upper bound for the limit (cf. \cite{Zyk}). This is analogous to the "easy" inequality in the classical Brauer--Siegel theorem. 

The interest of the question about the behaviour of $\log Z_K\left(\frac{1}{2}\right)$ can be in particular explained by its connection to the behaviour of the order of the Shafarevich--Tate group and the regulator of constant supersingular elliptic curves over function fields, the connection being provided by the Birch and Swinnerton--Dyer conjecture. In general, a similar question can be asked about the behaviour of these invariants in arbitrary families of elliptic curves. Some discussion on the problem is given in \cite{KT} (beware, however, that the proof of the main result there can not be seen as a correct one as the change of limits, which is a key point, is not justified). 

\end{remark}

\begin{proof}[Proof of theorem \ref{onehalf}:]
It follows from the basic inequality that the series defining $\log\zeta_{\{K_i\}}(s)$ converges absolutely for $\Re s \geq \frac{1}{2}.$ The function 
$\log\zeta_{\{K_i\}}(s)$ has a Dirichlet series expansion with positive coefficients, converging for $\Re s \geq \frac{1}{2}.$ Thus, from a standard theorem on Dirichlet series (cf. \cite[Lemma 5.56]{IK}), it must converge in some open domain $\Re s > \frac{1}{2}-\delta_0$ for $\delta_0 > 0,$ defining an analytic function there. It follows that in the same domain the series for $Z_{\{K_i\}}(s)$ converges. Taking any $\delta$ with $0<\delta<\delta_0$ we obtain:
\begin{eqnarray*}
\left|\sum_{f=1}^{N}\frac{f\phi_{r^f}}{r^{\frac{f}{2}}-1}-Z_{\{K_i\}}\left(\frac{1}{2}\right)\right|&=&\left|\sum_{f=N+1}^{\infty}\frac{f\phi_{r^f}}{r^{\left(\frac{1}{2}-\delta\right)f}-1}\cdot\frac{r^{\left(\frac{1}{2}-\delta\right)f}-1}{r^{\frac{f}{2}}-1}\right| \\
 &\leq& \left|\sum_{f=1}^{\infty}\frac{f\phi_{r^f}}{r^{\left(\frac{1}{2}-\delta\right)f}-1}\right|\cdot\frac{r^{\left(\frac{1}{2}-\delta\right)N}-1}{r^{\frac{N}{2}}-1} = O(r^{-\delta N}).
\end{eqnarray*}
This gives the necessary result. 
\end{proof}

\begin{proof}[Proof of the corollary $\ref{BSExpl}$:]
We use theorem \ref{onehalf} to obtain the necessary estimate much in the same spirit as in the proof of theorem \ref{onehalf} itself. Using the function field Brauer--Siegel theorem to find the value for $\kappa,$ we get:
\begin{eqnarray*}
\left|\sum_{f=1}^{N}\phi_{r^f}\log\frac{r^f}{r^f-1}-\kappa\right|&=&\left|\sum_{f=N+1}^{\infty}\frac{f\phi_{r^f}}{r^{\frac{f}{2}}-1}\cdot 
\frac{r^{\frac{f}{2}}-1}{f}\cdot\log\frac{r^f}{r^f-1}\right| \\
 &\leq& \left|\sum_{f=N+1}^{\infty}\frac{f\phi_{r^f}}{r^{\frac{f}{2}}-1}\right|\cdot \frac{r^{\frac{N}{2}}-1}{N}\cdot\log\frac{r^N}{r^N-1}\\
  &=& O(r^{-\delta N})\cdot O\left(\frac{r^{-\frac{N}{2}}}{N}\right).
\end{eqnarray*}
Indeed, $N\mapsto \frac{1}{N}(r^{\frac{N}{2}}-1)\log\frac{r^N}{r^N-1}$ is decreasing for $N\geq 2.$ The required estimate follows. 
\end{proof}

\begin{remark}
Actually, our method gives an easy and conceptual proof of the explicit version of the Brauer--Siegel theorem from \cite{Leb} (which is roughly speaking the statement of corollary \ref{BSExpl} with $\delta=0$). It shows that the rate of convergence in the Brauer--Siegel theorem essentially depends on how far to the left the limit zeta function $\zeta_{\{K_i\}}(s)$ is analytic. In the number field case we even save $\log^2 N$ in the estimate of the error term compared to what is proven in \cite{Leb}.
\end{remark}

\end{document}